\documentclass[12pt]{article}
\usepackage{amsmath}
\usepackage{amssymb}
\usepackage{amsthm}

\newtheorem{theorem}{Theorem}[section]
\newtheorem{lemma}[theorem]{Lemma}
\newtheorem{corollary}[theorem]{Corollary}

\theoremstyle{definition}

\newcommand{\N}{\mathbb{N}}
\newcommand{\Z}{\mathbb{Z}}

\newcommand{\fc}{\mathcal{F}}

\DeclareMathOperator{\ord}{ord}

\DeclareMathOperator{\supp}{supp}

\DeclareMathOperator{\so}{\mathsf{s}}
\DeclareMathOperator{\vo}{\mathsf{v}}

\DeclareMathOperator{\Do}{\mathsf{D}}

\DeclareMathOperator{\s}{\sigma}

\begin{document}
\title{Restricted inverse zero-sum problems in groups of rank two}
\author{Wolfgang A. Schmid\\
Centre de Math{\'e}matiques Laurent Schwartz\\
UMR 7640 du CNRS\\
{\'E}cole polytechnique\\
91128 Palaiseau cedex\\
France\\
\texttt{wolfgang.schmid@math.polytechnique.fr}}

\date{}
\maketitle


\begin{abstract}
Let $(G,+)$ be a finite abelian group. Then, $\so(G)$ and $\eta(G)$ denote the smallest integer $\ell$
such that each sequence over $G$ of length at least $\ell$ has a subsequence whose terms sum to $0$ and whose length is equal to and at most, resp., the exponent of the group.
For groups of rank two, we study the inverse problems associated to these constants, i.e., we investigate the structure of sequences of length $\so(G)-1$ and $\eta(G)-1$ that do not have such a subsequence.
On the one hand, we show that the structure of these sequences is in general richer than expected.
On the other hand, assuming a well-supported conjecture on this problem for groups of the form $C_m \oplus C_m$,
we give a complete characterization of all these sequences for general finite abelian groups of rank two.
In combination with partial results towards this conjecture, we get unconditional characterizations in special cases.
\end{abstract}

\section{Introduction}

The investigation of the following type of problem was initiated in the 1960's by the work of P.~Erd\H{o}s, A.~Ginzburg, and A.~Ziv~\cite{erdosginzetal61}.
Let $G$ be an additive finite abelian group. Determine the smallest integer $\ell$ such that each sequence over $G$ of length at least $\ell$ has a subsequence the sum of whose terms equals $0 \in G$ and that fulfills some additional property; in particular, restrictions on the length of the subsequence were considered (see Section \ref{sec_prel} for a more formal definition).

In the present paper, we are specifically interested in the constants
$\so(G)$ and $\eta(G)$ that arise when imposing the condition that the subsequence
has length equal to the exponent of the group and length at most the exponent of the group, respectively.
Together, with the constant $\mathsf{ZS}(G)$ (subsequence of length equal to the order of the group) and
the Davenport constant $\Do(G)$ (no restriction on the length of the subsequence, besides the trivial one that the length is not zero, to exclude the empty sequence) these are the most classical constants of this form.
The constant $\so(G)$ is a generalization (to general groups) of the original problem consider for cyclic groups in \cite{erdosginzetal61}, first investigated in detail by H.~Harborth~\cite{harborth73}.
The constant $\eta(G)$ was first investigated by P.~van Emde Boas \cite{vanemdeboas69} and J.E.~Olson \cite{olson69_2}, as a key-tool in the investigation of the Davenport constant for groups of rank two.
Parallel to the direct problem of determining the value of these constants, the associated inverse problems, i.e., the problem of determining the structure of the longest sequences that do not have a subsequence of the above mentioned type, received considerable attention as well.

We refer to the recent paper of Y.~Edel et al.~\cite{edeletal07} for a detailed exposition of the history and
applications of these two constants, among others in discrete geometry and non-unique factorization theory.
For various results on these constants and other related problems see the survey articles \cite{caro96,gaogeroldingersurvey} and the monograph \cite{geroldingerhalterkochBOOK}, in particular Chapter 5.

Here, these inverse problems for general finite abelian groups of rank two are investigated.
We give a short summary of the present state of knowledge on these invariants (in general), to illustrate that
to consider this problem for groups of rank two is a natural choice.
The direct problems for groups of rank at most two are solved (cf.~Theorem \ref{thm_dir} and the references there). Moreover, for cyclic groups, answers to the inverse problems are well-known (cf.~Theorem \ref{thm_invcyc}), yet the refined problem of determining the structure of shorter sequences
without subsequences of the respective form received considerable attention in the recent literature. We refer to, e.g., \cite{yuan07,savchevchen07,savchevchen08,nguyen09,balandraud} for results of this form; note that for cyclic groups---and only in this case---the inverse problem associated to $\eta(G)$ is, for immediate reasons, identical to the one associated to $\Do(G)$ (for recent investigations on the inverse problem associated to the Davenport constant for groups of rank two see \cite{reiherneu,GGGinverse3,WASinverse2}). Whereas, for groups of rank at least three, both the direct and the inverse problem are in general wide open (see, e.g., \cite{alondubiner95,edeletal07,edel08,schmidzhuang} for partial results and bounds), as is the problem of determining the Davenport constant (see, e.g., \cite{bhowmikschalge07} for a recent contribution).

For groups of the form $C_m^2$, there is a well-known and well-supported conjecture regarding the answers to the inverse problems (see Section \ref{sec_res} for details).
Yet, for general groups of rank two the situation was unclear.
Some examples of extremal sequences have been established (see, e.g., \cite[Proposition 5.7.8]{geroldingerhalterkochBOOK} and \cite{edeletal07}, in particular Lemma 3.2 and the remarks after Lemma 2.3, for constructions valid for more general groups as well). Our investigations show that these constructions are not exhaustive in an essential way; some expansion on the known constructions are immediate---the goal in the just mentioned work was not to give a complete list of examples---yet beyond these immediate modifications we exhibit aspects that were not noticed before. In particular for the problem associated to $\so(G)$, the structure of sequences can be richer than expected. More specifically, it was conjectured (see \cite[Conjecture 7.1]{gaogeroldingersurvey}) that, for $G$ a finite abelian group, each sequence $S$ over $G$ of length $|S|=\so(G)-1$ that has no zero-sum subsequence of length equal to  $\exp(G)$ contains some element $\exp(G)-1$ times.
Our investigations yield an example showing that groups of rank two in general do not have this property (cf.~Corollary \ref{cor_exp-1}).

Moreover, and this is the main part of the present work, we reduce the problem of solving the inverse problems for general finite abelian groups of rank two to the respective inverse problems for groups of the form $C_m^2$.
Assuming that the above mentioned conjecture for the groups $C_m^2$ holds true, we get a complete solution for rank two groups (see Theorem \ref{thm_main}). And, in combination with partial results towards this conjecture, we obtain unconditional
results in special cases (see Corollary \ref{cor_uncond}). In fact, due to a very recent result of Ch.~Reiher \cite{reiherneu}, the result regarding the inverse problem associated to $\eta(G)$ becomes unconditional.

\section{Preliminaries}
\label{sec_prel}
We recall some notation and terminology (following \cite{gaogeroldingersurvey} and \cite{geroldingerhalterkochBOOK}).

We denote by $\Z$ the set of integers, and  by $\N$ and $\N_0$ the positive and non-negative integers, respectively.
We denote by $[a,b]=\{z \in \Z \colon a\le z\le b\}$ the interval of integers.
For $k \in \Z$ and $m \in \N$, we denote by $[k]_m$ the smallest non-negative integer that is congruent to $k$ modulo $m$.

Let $(G,+)$ denote a finite abelian group; throughout, we use additive notation for abelian groups.
For a subset $G_0 \subset G$, we denote by $\langle G_0\rangle$ the subgroup generated by $G_0$.
We say that elements $e_1, \dots, e_r \in G\setminus \{0\}$ are independent
if $\sum_{i=1}^rm_ie_i=0$ with $m_i \in \Z$ implies that $m_ie_i=0$ for each $i \in [1,r]$.
We say that a subset of $G$ is a basis if it generates $G$ and its elements are independent.
For $n \in \N$, we denote by $C_n$ a cyclic group of order $n$.
For each finite abelian group $G$, there exist uniquely determined $1< n_1 \mid \dots \mid n_r$ such that
$G\cong C_{n_1}\oplus \dots  \oplus C_{n_r}$; then $r$ is the rank of $G$ and  $\exp(G)=n_r$ the exponent of $G$.

We denote by $\fc(G)$ the, multiplicatively written,  free abelian monoid over $G$, that is,
the monoid of all formal commutative products
\[S=\prod_{g\in G} g^{\vo_g(S)}\]
with $\vo_g(S)\in \N_0$.
An element $S\in \fc(G)$ is called a sequence over $G$; strictly speaking, this is not a finite sequence in the usual sense---as the ordering of the terms is disregarded---yet for the  questions considered in our context the ordering is irrelevant anyway, while this formal framework has several advantages.
We refer to $\vo_g(S)$ as the multiplicity of $g$ in $S$. Moreover, $\s(S)=\sum_{g \in G} \vo_g(S)g\in G$ is called the sum of $S$,
 $|S|=\sum_{g \in G} \vo_g(S)\in \N_0$ the length of $S$, and $\{g \in G\colon \vo_g(S) > 0\}\subset G$ the support of $S$.
We say that a sequence $S$ (over $G$) is short if $|S|\in [1,\exp(G)]$.

We denote the unit element of $\fc(G)$ by $1$ and call it the empty sequence.
If $T \in \fc(G)$ and $T \mid S$ (in $\fc(G)$), then we call $T$ a subsequence of $S$.
Moreover, we denote by $T^{-1}S$ the unique sequence $R$ with $RT=S$.
The sequence $S$ is called zero-sum free, if $\s(T)\neq 0$ for each $1\neq T \mid S$.

Having more notation at hand, we restate the definition of the invariants mentioned in the introduction in a more formal way, and mentioned a fourth one which we need in one of our arguments.
For a finite abelian group $G$, let $\ell\in \N$ be minimal with the property that each $S\in \fc(G)$ with $|S| \ge \ell$ has a subsequence $T\mid S$ such that $\s(T)=0$ and
\begin{itemize}
\item $|T|=\exp(G)$;  $\ell$ is denoted by $\so(G)$.
\item $|T| \in [1, \exp(G)]$;  $\ell$ is denoted by $\eta(G)$.
\item $|T|\ge 1$;  $\ell$ is denoted by $\Do(G)$.
\item $|T|=k\exp(G)$ for some $k\in \N$;  $\ell$ is denoted by $\so_{\exp(G)\N}(G)$.
\end{itemize}

For each map $f: G \to G'$ between finite abelian groups, there exists a unique extension to a monoid homomorphism $\fc(G) \to \fc(G')$, which we denote by $f$ as well. And, if $f$ is group homomorphism, then $\s(f(S))= f(\s(S))$ for each $S \in \fc(G)$. Moreover, for $h \in G$ and $S \in \fc(G)$, let $s_h: G \to G$ denote the map defined via $g \mapsto g+h$, and let $h+S$ denote $s_h(S)$.

\section{Main result and Discussion}
\label{sec_res}

As mentioned in the Introduction, we reduce the inverse problem for general groups of rank two to the inverse problem for groups of the form $C_m^2$ for which these problems are well-understood.

We recall two related key-notions for $C_m^2$.
Let $m \in \N$. The group $G=C_m^2$ is said to have
\begin{itemize}
\item Property \textbf{C} if each $S\in \fc(G)$ of length $|S|=\eta(G)-1$ that has no short zero-sum subsequence  equals $T^{\exp(G)-1}$ for some $T \in \fc(G)$.
\item Property \textbf{D} if each $S\in \fc(G)$ of length $|S|=\so(G)-1$ that has no zero-sum subsequence of length $\exp(G)$ equals $T^{\exp(G)-1}$ for some $T \in \fc(G)$.
\end{itemize}
Property \textbf{C} was first formulated and investigated by P.~van Emde Boas \cite{vanemdeboas69}; to be precise, he considered a slightly weaker yet essentially equivalent property (cf.~\cite[Lemma 4.7]{gao00a} for details). And, Property \textbf{D} was introduced by W.D.~Gao~\cite{gao00}.
It is well-known that if $C_m^2$ has Property \textbf{D}, than it has Property \textbf{C} (see \cite[Lemma 3.3]{gaogeroldinger03a}).

It is conjectured that for each $m\in \N$ the group $C_m^2$ has Property \textbf{D} and (thus) Property \textbf{C} (see the two just mentioned papers and, e.g., \cite[Conjecture 7.2]{gaogeroldingersurvey}).
And, very recently Ch.~Reiher \cite{reiherneu} proved that $C_m^ 2$ has Property \textbf{C} for each $m \in \N$.

We recall some partial results on Property \textbf{D}.
By a result of W.D.~Gao~\cite{gao00}, the property is multiplicative, i.e., if for $m,n\in \N$ both $C_m^2$ and $C_n^2$ have Property \textbf{D} so does $C_{mn}^2$; and an analogous assertion is known for Property \textbf{C} (also see \cite[Theorem 3.2]{WAS20} for a version of this result for arbitrary rank), reducing the problem of establishing this property for $C_m^2$  to the case where $m$ is prime. Moreover, it is known to hold true for small $m$, namely, for $m \le 10$ (see \cite{gao00,surythangadurai02}).

We formulate our main result. In combination with the above mentioned results
it yields unconditional answers to the inverse problems in certain cases, cf.~Corollary \ref{cor_uncond} for a formal statement; indeed, by Ch.~Reiher's result \cite{reiherneu} the part regarding $\eta(G)$ holds unconditionally (yet, to highlight the parallelity of the two assertions and to reflect the actual content of this paper, we formulate the result in this way).

\begin{theorem}
\label{thm_main}
Let $G$ be a finite abelian group of rank two, say, $G\cong C_m \oplus C_{mn}$ with  $m, n \in \N$ and $m \ge 2$.
Let $\{e_1,e_2\}$ be a basis of $G$ with $\ord e_2= mn$, and let $\{g_1,g_2\}$ be a generating set of $G$ with $\ord g_2=mn$.
\begin{enumerate}
\item The following sequences have length $\eta(G)-1$ and no short zero-sum susbequence.
\begin{enumerate}
\item \(e_1^{m-1} e_2^{sm  -1} (-xe_1+e_2)^{(n+1-s)m-1}\) with $\gcd\{x,m\}=1$ and $s\in [1,n]$.
\item \(g_1^{m-1} g_2^{mn -1} (-g_1+g_2)^{m-1}.\)
\end{enumerate}
If $C_m^2$ has Property \textbf{C}, then each sequence $S\in \fc(G)$ with $|S|=\eta(G)-1$ and no short zero-sum subsequence is of this form (for some basis or generating set, resp., with the above properties).
\item The following sequences have length $\so(G)-1$ and no zero-sum subsequence of length $\exp(G)$.
\begin{enumerate}
\item \(g^{t m-1 }(e_1+g)^{(n+1 - t ) m-1} (e_2+g)^{sm  -1} (-xe_1+e_2+g)^{(n+1-s)m-1}\) where $\gcd\{x,m\}=1$, $s,t \in [1,n]$, and $g \in G$.
\item \(g^{mn-1}(g_1+g)^{m-1} (g_2+g)^{mn -1} (-g_1+g_2+g)^{m-1}\) where $g\in G$.
\end{enumerate}
If $C_m^2$ has Property \textbf{D}, then each sequence $S\in \fc(G)$ with $|S|=\so(G)-1$ and no zero-sum subsequence of length $\exp(G)$ is of this form (for some basis or generating set, resp., with the above properties).
\end{enumerate}
\end{theorem}

We point out that to avoid technicalities Theorem \ref{thm_main} is  formulated  in such a way that neither the examples of sequences are mutually exclusive (e.g., we additionally could impose the condition $x \le m/2$, cf.~Lemma \ref{lem_noshort}) nor cover all representations of sequences with respect to ``natural'' bases or generating sets (e.g., the sequence $e_1^{m-1}e_2^{m-1}(-xe_1+xe_2)^{m-1}$ over $C_m^2$ with $\gcd\{x,m\}=1$ has no short zero-sum subsequence and at first might seem to be of a different type, yet by considering it with respect to the basis $\{-xe_1+xe_2, e_1\}$ it is readily seen that it is covered by our result).
Moreover, we note that (b), in both cases, is redundant for $n=1$, since then there are no generating sets with two elements that are not a basis. Yet, also in this case, the assertion of our result is more precise than
what is immediate by assuming Properties \textbf{C} and \textbf{D}, respectively.

We end this section by stating, in a formal way, two points that we informally mentioned before.

The following result summarizes the present state of knowledge regarding a full and unconditional solution of the inverse problems associated to  $\eta(G)$ and $\so(G)$ for groups of rank two.

\begin{corollary}
\label{cor_uncond}
Let $G$ be a finite abelian group of rank two, say, $G\cong C_m \oplus C_{mn}$ with  $m, n \in \N$ and $m \ge 2$. Let $S \in \fc(G)$.
\begin{enumerate}
\item The sequences $S$ has length $\eta(G)-1$ and no short zero-sum subsequence if and only if
\begin{itemize}
\item there exist a basis $\{e_1,e_2\}$ of $G$ with $\ord e_2= mn$, $x \in \N$ with $\gcd\{x,m\}=1$, and $s \in [1,n]$ such that
\(S = e_1^{m-1} e_2^{sm  -1} (-xe_1+e_2)^{(n+1-s)m-1},\) or
\item there exists a generating set $\{g_1,g_2\}$ of $G$ with $\ord g_2=mn$ such that  \(S = g_1^{m-1} g_2^{mn -1} (-g_1+g_2)^{m-1}.\)
\end{itemize}
\item Suppose $m$ is not divisible by a prime strictly  greater than $7$.
The sequences $S$ has length $\so(G)-1$ and no zero-sum subsequence of length $\exp(G)$ if and only if
\begin{itemize}
\item there exist a basis $\{e_1,e_2\}$ of $G$ with $\ord e_2= mn$, $x \in \N$ with $\gcd\{x,m\}=1$, $s,t \in [1,n]$, and $g \in G$ such that
\(S=g^{t m-1 }(e_1+g)^{(n+1 - t ) m-1} (e_2+g)^{sm  -1} (-xe_1+e_2+g)^{(n+1-s)m-1},\)  or
\item there exists a generating set $\{g_1,g_2\}$ of $G$ with $\ord g_2=mn$ and $g \in G$ such that \(S=g^{mn-1}(g_1+g)^{m-1} (g_2+g)^{mn -1} (-g_1+g_2+g)^{m-1}.\)
\end{itemize}
\end{enumerate}
\end{corollary}
\begin{proof}
1. By \cite{reiherneu} we know that each $m \in \N$ has Property \textbf{C}
(recall that Property \textbf{C} is implied by Property \textbf{B}, see \cite[Theorem 10.7]{gaogeroldinger99}, and that Property \textbf{C} is multiplicative, see \cite{gao00}).
Thus, the condition in Theorem \ref{thm_main}.1 is fulfilled and the claim follows.

\noindent
2. By \cite{gao00} and \cite{surythangadurai02} we know that if $m$ has no prime divisor strictly greater than $7$, then $m$ has Poperty \textbf{D}.
Thus, the condition in Theorem \ref{thm_main}.2 is fulfilled and the claim follows.
\end{proof}

Moreover,  as a consequence of Theorem \ref{thm_main}.2, we get that, for groups of rank two, the structure of sequences of length $\so(G)-1$ without zero-sum subsequence of length $\exp(G)$
can be more complicated than expected, though (provided Property \textbf{D} holds) for groups of rank two only slighly so.
In particular, we can answer (negatively) \cite[Conjecture 7.1]{gaogeroldingersurvey}.
\begin{corollary}
\label{cor_exp-1}
Let $G=C_m\oplus C_{mn}$ with $m \ge 2$ and $n \ge 3$. There exists a sequence $S\in \fc(G)$ of length $\so(G)-1$ that has no zero-sum subsequence of length $\exp(G)$ yet $\vo_{g}(S)< \exp(G)-1$ for each  $g\in G$.
\end{corollary}
\begin{proof}
Clear, by Theorem \ref{thm_main}.2.a with $s,t \in [2,n-1]$.
\end{proof}

\section{Proof of Theorem \ref{thm_main}}

In this section we prove Theorem \ref{thm_main}.
First, we recall and establish some auxiliary results and then turn to the actual details of the proof.

\subsection{Auxiliary results}

In the following theorem, we recall the answers to the direct problems for groups of rank at most two; in part, they are classical, yet
the results on $\so(G)$ and $\so_{\exp(G) \N}(G)$ for groups of rank two were obtained only fairly recently (see \cite[Theorem 5.8.3]{geroldingerhalterkochBOOK} building on crucial contributions by \cite{reiher07,savchevchen05}, and \cite[Theorem 6.5]{gaogeroldingersurvey}, respectively).

\begin{theorem}
\label{thm_dir}
Let $m,n\in \N$ and $G=C_m\oplus C_{mn}$.
Then $\Do(G)=m+mn-1$, $\eta(G)=2m +mn -2$, $\so_{\exp(G) \N}(G)=m+2mn-2$, and $\so(G)=2m+2mn-3$.
\end{theorem}

For cyclic groups, solutions to the inverse problems are well-known and as discussed in the Introduction meanwhile refined results---valid for shorter sequences---are known (cf., e.g., \cite[Theorems 4.3 and 7.5]{gaogeroldingersurvey} for results containing the result below and detailed references).

\begin{theorem}
\label{thm_invcyc}
Let $n \in \N$ and $S \in \fc(C_n)$.
\begin{enumerate}
\item Suppose $|S|= \eta(C_n)-1  = \Do(C_n)-1$. Then $S$ has no (short) non-empty zero-sum subsequence if and only if $S=e^{n-1}$ for some $e\in C_n$ with $\langle e \rangle=C_n$.
\item Suppose $|S|=\so(C_n)-1$. Then $S$ has no zero-sum subsequence of length $n$ if and only if $S=g^{n-1}(g+e)^{n-1}$ for $g, e\in C_n$ with $\langle e \rangle=C_n$.
\end{enumerate}
\end{theorem}

In the following lemma, we collect some facts that we use and are essentially known (cf.~\cite[Lemma 2.2]{edeletal07} and \cite[Theorem 2]{gaothangadurai03}).
\begin{lemma}
\label{lem_general}
Let $G$ be a finite abelian group, $g \in G$, and $S \in \fc(G)$. Furthermore, let $n \in \N$ such that $\exp(G)\mid n$.
\begin{enumerate}
\item $S$ has a zero-sum subsequence of length $n$ if and only if $g+S$
has a zero-sum subsequence of length $n$.
\item If $S$ has no short zero-sum subsequence,  then, for $v \in [0, \exp(G)-1]$, $g^v(g+S)$ has no zero-sum subsequence of length $\exp(G)$.
\item If $\vo_{g}(S)\ge \lfloor (\exp(G)-1)/2 \rfloor$ and $S$ has no zero-sum subsequences of length $\exp(G)$, then $S$ has a subsequence $T$ of length at least $|S|-\exp(G)+1$ such that $(-g)+T$ has no short zero-sum subsequence.
\end{enumerate}
\end{lemma}

The following lemma, which for prime $m$ can be found in \cite[Section 5]{vanemdeboas69}, is needed in the proof of Theorem \ref{thm_main}.1; it gives information on the sequence $T$ appearing in the formulation of Property \textbf{C}.

\begin{lemma}
\label{lem_noshort}
Let $m \in \N$ with $m \ge 2$.
Let $T^{m-1}\in \fc (C_m^2)$ be a sequence of length $3m-3$ that has no short zero-sum subsequence.
Then $T = f_1f_2(-xf_1 +f_2)$ for a basis $\{f_1, f_2\}$ of $C_m^2$ and some $x\in [1,m-1]$ with $\gcd\{x,m\}=1$; moreover, $x \le m/2$.
In particular, for each $f \in \supp(T)$, the sequence $(f^{-1}T)^{m-1}$ is zero-sum free.
\end{lemma}
\begin{proof}
Obviously $|\supp(T)|=3$.
By \cite[Lemma 4.4]{WAS20}, the sequence $T^{m-1}$ has a minimal zero-sum subsequence $U$ of length $2m-1$.
By \cite[Theorem 1]{WAS15} and \cite[Proposition 4.1.2]{gaogeroldinger03a}, we know that $U= e_1^{m-1}\prod_{i=1}^{m} (a_ie_1+e_2)$ for some basis $\{e_1, e_2\}$ of $C_m^2$. Thus, $T=e_1(ae_1+e_2)(be_1+e_2)$ with distinct $a, b \in [0,m-1]$, say $a> b$.
Obviously both $\{e_1, ae_1 + e_2\}$ and $\{e_1, be_1 + e_2\}$ are a basis of $C_m^2$, and by \cite[Corollary 1]{WAS15}, $\{b e_1 + e_2, ae_1 + e_2\}$ is a basis as well, which implies that $\gcd\{ a - b , m \} = 1$.
Furthermore,  $be_1 + e_2 = -(a-b)e_1 + (ae_1 + e_2)$ and $a e_1+e_2 = -(m+b-a)e_1 + (be_1+e_2)$.
Since $0 \le \min \{a-b, (m+b-a)\}\le m/2$, the claim follows.
The ``in particular''-statement is a direct consequence of the explicit description.
\end{proof}

The following technical result is needed in the proof of Theorem \ref{thm_main}.2.
\begin{lemma}
\label{lem_2m}
Let $m \in \N$ with $m \ge 2$.
Let $T^{m-1}\in \fc (C_m^2)$ be a sequence of length $4m-4$ that has no zero-sum subsequence of length $m$.
Then for each $f \in \supp(T)$, the sequence $(f^{-1}T)^{m-1}$ has no zero-sum subsequence of length $2m$.
\end{lemma}
\begin{proof}
Let $f \in \supp(T)$. By Lemma \ref{lem_general}, the sequence $(-f+(f^{-1}T)^{m-1})$ has no short zero-sum subsequence.
By Theorem \ref{thm_dir}, each zero-sum sequence over $C_m^2$ of length $2m$ is not minimal, implying it is the product of  two non-empty zero-sum sequences, and at least one of them is short.
Thus,  $(-f+(f^{-1}T)^{m-1})$ has no zero-sum subsequence of length $2m$, which by Lemma \ref{lem_general} implies that  $(f^{-1}T)^{m-1}$ has no zero-sum subsequence of length $2m$.
\end{proof}

\subsection{Details of the proof}

First, we briefly indicate that the listed sequences indeed have the claimed properties.

Then, we solve the inverse problem associated to $\eta(G)$ conditional on Property \textbf{C}, and do likewise for the one associated to $\so(G)$ conditional on Property \textbf{D}.

\subsubsection{Establishing the properties of the sequences}
In both cases, the statements regarding the length of the sequences are immediate by Theorem \ref{thm_dir}.

Let $S$ be a sequence as in the formulation of the respective result.

\noindent
(1) We have to show that $S$ has no short zero-sum subsequence.
Let $1\neq T \mid S$ be a zero-sum subsequence.
Suppose $S$ is of the form given in (a). It is clear that $\vo_{e_2}(T)+\vo_{-xe_1+e_2}(T)$ is not $0$ and divisible by $\ord e_2=mn$; thus, it is equal to $mn$. This implies that $m \nmid \vo_{-xe_1+e_2}(T)$. Consequently, $\vo_{e_1}(T)\neq 0$ and $|T|> mn$.
Now, suppose $S$ is of the form given in (b). Since $ag_1 \notin \langle g_2 \rangle$ for $|a| \in [1,m-1]$, we have $\vo_{g_1}(T)=\vo_{-g_1+g_2}(T)$. Thus, $\vo_{-g_1+g_2}(T)+\vo_{g_2}(T)$ is divisible by $\ord g_2=mn$. So, we have $\vo_{-g_1+g_2}(T)\neq 0$ and $|T|> mn$. Thus, $S$ has no short zero-sum subsequence.

\noindent
(2) By Lemma \ref{lem_general} we may assume that $g=0$.
If $S$ is of the form given in (b), then we know by (1) that the sequence $g_1^{m-1} g_2^{mn -1} (-g_1+g_2)^{m-1}$ has no short zero-sum subsequence.
Similarly, if $S$ is of the form given in (a), we note that by the argument in (1), each zero-sum subsequence $T$ of $S$ with $\vo_{e_2}(T)+\vo_{-xe_1+e_2}(T)>0$ is not short.
Since neither $0^{mn-1}$ nor $0^{t m -1 }e_1^{(n +1- t ) m-1}$ has a zero-sum subsequence of length $mn$, it follows in both cases that $S$ has no zero-sum subsequence of length $mn$.

\subsubsection{Proof of Theorem \ref{thm_main}.1}
We start with some observations.

The case $n=1$ is an immediate consequence of Lemma \ref{lem_noshort}.

We thus assume $n \ge 2$, that is  $G\cong C_{m}\oplus C_{mn}$ with $m\ge 2$ and $n \ge 2$. Furthermore, let $H = \{mg \colon g \in G\} \cong C_n$ and let $\varphi : G \to G/H$ be the canonical map; we have $G/H \cong C_m^2$. We apply the Inductive Method, as in \cite[Section 8]{gaogeroldinger99}, with the exact sequence
\[0 \to H \hookrightarrow G \overset{\varphi}{\to} G/H \to 0,\]
partly using arguments similar to those in  \cite{gao00} and \cite{WAS20}.

Now, let $S \in \fc(G)$ be a sequence of length $\eta(G)-1$ with no short zero-sum subsequence.
We have to show that $S$ is of the claimed form.

We start our argument by showing that $|\supp(S)|=3$ and  already  obtain a somewhat more precise result on $S$ in the process of doing so (see \eqref{eq_C}).
Since $\eta(G/H)= 3m-2$ and $|S|= (n-1)m + 3m-3$, we know that there exist subsequences $S_1, \dots, S_{n-1}$ of $S$ such that $\prod_{i=1}^{n-1} S_i \mid S$ and each $\varphi(S_i)$ is a short zero-sum sequence over $G/H$, i.e., $\s(\varphi(S_i))=0$ and $|\varphi(S_i)|\le m$. Let $R\in \fc(G)$ such that $S= R \prod_{i=1}^{n-1} S_i $.
If $\s(S_1)\dots \s(S_{n-1})$ has a (short) zero-sum subsequence, say $\s( \prod_{i\in I} \s(S_i))=0$ for some $\emptyset \neq I \subset [1,n-1]$, then $\s( \prod_{i\in I} S_i)=0$ and $|\prod_{i\in I} S_i|\le |I|m \le mn$, a contradiction.
Thus, this sequences has no zero-sum subsequence and consequently by Theorem \ref{thm_invcyc} $\s(S_1)= \dots =\s(S_{n-1})= e$ where $\langle e \rangle = H$.
Moreover, by the above reasoning it follows that $\varphi(R)$ does not have a short zero-sum subsequence.
Thus, $|R| \le  3m-3$ and it follows that $|S_i|= m$ for each $i\in [1,n-1]$ and $|R|=3m-3$. Since $C_m^2$ has Property \textbf{C}, $\varphi(R)=T^{m-1}$ for
some sequences $T \in \fc(G/H)$ with $|T|=3$.

We show that $\supp(\varphi(S))= \supp (\varphi (R))$; more precisely, we show that  each $\varphi(S_i)$ is equal to $f^m$ for some $f \in \supp(T)$.
Assume this is not the case, say $\varphi(S_1)$ is not of this form.
We show that $\varphi(S_1R)$ is divisible by the product of two short zero-sum sequences, which by the above argument yields a contradiction.

First, assume  $f\mid \varphi(S_1)$ for some $f \in \supp(T)$. Then $f^m\mid \varphi(S_1R)$ and $f^{-m}\varphi(S_1R)= (f^{-1}\varphi(S_1)) (f^{-1}T)^{m-1}$.
Since by assumption $\varphi(S_1)\neq f^m$, it follows that $|\supp(f^{-1}\varphi(S_1))|\ge 2$. Thus, $f^{-m}\varphi(S_1R)$ contains an element with multiplicity at least $m$ or  its support contains at least $4$ distinct elements; in both cases, in the latter using the fact that its length is $3m-3$ and $C_m^2$ has Property \textbf{C},  it has a short zero-sum subsequence.

Second, assume $\supp(\varphi(S_1)) \cap \supp(\varphi(R))= \emptyset$. Let $f' \mid \varphi(S_1)$. The sequence $f' \varphi(R)$ has a short zero-sum subsequence $U$. We know that $f' \mid U$ and thus, since $f' \notin \supp(T)$, we have $f^{m-1}\nmid U$ for each $f \in \supp(T)$. Therefore, $\supp(U^{-1}f' \varphi(R))= \supp(\varphi(R))$ and $|\supp(U^{-1}\varphi(S_1R))|\ge 4$. This implies the existence of a short zero-sum subsequence of $U^{-1}\varphi(S_1R)$.

Now, we show that $|\varphi^{-1}(f)\cap \supp (S)|= 1$ for each $f \in \supp(\varphi(S))$.
Assume not. Let $g, g' \in \supp (S)$ be distinct elements such that $\varphi(g)= \varphi(g')=f$.

First, suppose $f \in \supp(\varphi(\prod_{i=1}^{n-1}S_i))$.
We may assume that $g'\mid R$ and, say,  $g \mid S_1$.
We have $\s(S_1)=e$. Let $S_1'=g^{-1}g'S_1$ and $R'= g'^{-1}gR$.
As above, we know that $\s(S_1')\s(S_2)\dots \s(S_{n-1})$ has no zero-sum subsequence. Thus, it follows that $\s(S_1')=e$ (for $n=2$ this is the only generating element of $H$). Yet, $\s(S_1') = \s(S_1)+g'-g\neq \s(S_1)$, a contradiction.

Second, suppose $g, g' \in \supp(R)$. This implies $m \ge 3$.
The sequence $g^{-1}R$ has a subsequence $V$ such that $\varphi(V)$ is a minimal zero-sum sequence, thus in particular $|V|\le 2m-1$.
Since $(f^{-1}T)^{m-1}$, for $f \in \supp (T)$,  is zero-sum free (see Lemma \ref{lem_noshort}), we have $\supp(\varphi(V))= \supp(\varphi(R))$ and thus we may assume that $g'\mid V$. If $\s(U)\neq e$, then $\s(S_1)\dots \s(S_{n-1})\s(U)$ has a zero-sum subsequence of length at most $n-1$, yielding a zero-sum sequence of $S$ of length at most $(n-2)m+ 2m-1\le mn$, a contradiction. Thus, we have $\s(V)=e$. Yet, the same is true for $\s(g'^{-1}gV)$, a contradiction.

So, we know that
\begin{equation}
\label{eq_C}
S= g_1^{s_1m -1}g_2^{s_2 m -1}g_3^{s_3m -1}
\end{equation}
 with $s_i \in [1, n]$ and $s_1+s_2+s_3=n+2$, in particular $|\supp(S)|=3$.

We recall that by Lemma \ref{lem_noshort} $\supp (\varphi(S))= \supp(\varphi(R))= \{f_1, f_2, -x f_1 + f_2\}$ for a basis $f_1, f_2$ and some $x\in [1, m]$  with $\gcd (x, m)=1$ and $x \le m/2$. Say, $\varphi(g_i)=f_i$ for $i\in [1,2]$.
We note that if $s_i\ge 2$, then $mg_i= \s(g_i^m)=e$, in particular $\ord g_i = mn$.

For $a \in [1,m-1]$, let $R_a= g_1^{[xa]_m} g_2^{m-a} g_3^a$. Then $\varphi(R_a)$ is a zero-sum subsequence of $\varphi(R)$ of length at most $2m-1$.
Thus, as above, we conclude $\s(R_a)=e$ for each $a\in [1, m-1]$.

Considering $a=1$ we have
\begin{equation}
\label{eq_a=1}
x g_1 + (m-1) g_2 + g_3 = e
\end{equation}
and considering $a= m-1$  we have
\begin{equation}
\label{eq_a=m-1}
(m-x) g_1 + g_2 + (m-1)g_3 = e.
\end{equation}

Now, we assume $m \ge 3$ and complete the argument. At the end we consider $m=2$.
Considering $a=2$, we get
\begin{equation}
\label{eq_a=2}
2x g_1 + (m-2)g_2+ 2g_3=e.
\end{equation}
Thus, considering the difference of \eqref{eq_a=2} and \eqref{eq_a=1} we get
\begin{equation}
\label{eq_g1g2g3}
xg_1 - g_2 + g_3=0.
\end{equation}
Moreover, considering the difference of \eqref{eq_a=2} and two times \eqref{eq_a=1} we get $mg_2=e$, in particular $\ord g_2=mn$, and combining this with the sum of \eqref{eq_a=1} and \eqref{eq_a=m-1} we get $mg_1+mg_3=e$.
Note that $\{ g_i, g_2\}$ for $i\in \{1,3\}$ is a generating set of $G$, since $bg_i \notin \langle g_2\rangle$ for $b\in [1, m-1]$.

First, suppose $x \neq 1$. Then $1\le \lceil m/x \rceil< m$. Let $r= [\lceil m/x \rceil x]_m= \lceil m/x \rceil x-m$.
Considering $a=\lceil m/x \rceil$, we get
\begin{equation} rg_1  + (m-\lceil m/x \rceil)g_2 +\lceil m/x \rceil g_3=e\end{equation}
and thus
$rg_1  -\lceil m/x \rceil g_2 + \lceil m/x \rceil g_3=0$ and
$(\lceil m/x \rceil x-m)g_1  - \lceil m/x \rceil g_2 +\lceil m/x \rceil g_3=0$.
Using \eqref{eq_g1g2g3}, we get $mg_1 = 0$ and thus $s_1=1$.  Moreover, it follows that $\{g_1, g_2\}$ is a basis of $G$ and by \eqref{eq_g1g2g3} the sequence is of the form given in (a).

Second, suppose $x=1$. If $s_1=s_3=1$, the sequence is of the form given in (b), since by \eqref{eq_g1g2g3} $g_3= -g_1+g_2$.
If $s_3\ge 2$, then $mg_3=e$ and $mg_1= 0$, implying  that  $\{g_1, g_2\}$ is a basis of $G$, completing the argument.
Similarly, if $s_1 \ge 2$, then $mg_1= e$ and $mg_3=0$. Now, $\{g_3, g_2\}$ is a basis of $G$ and $g_1= -g_3+g_2$, again completing the argument.

Finally, we suppose $m=2$.  Then $x=1$ and $g_1+g_2+g_3=e$.
Let $\{i, j, k\}=\{1,2,3\}$ such that $s_i \ge 2$. Then $2g_i= e$ and $-g_i+g_j+g_k=0$.
If $s_j\ge 2$, then $2g_j=e$, and $2g_k=0$. It follows that $\{g_k, g_i\}$ is a basis of $G$ and $g_j = -g_k +g_i$.
If $s_j=s_k=1$, we have $g_j= -g_k+g_i$ and $\{g_k, g_i\}$ is a generating set.
This completes the proof of part 1.

\subsubsection{Proof of Theorem \ref{thm_main}.2}

As for part 1, the case $n=1$ is immediate, by Lemma \ref{lem_general} and part 1. We thus assume again $n\ge 2$, and use the same exact sequence as in the proof of part 1. Also, other parts of the argument are similar to the one for part 1, we keep those parts brief.

Let $S \in \fc(G)$ be a sequence of length $\so(G)-1$ with no zero-sum subsequence of length $\exp(G)$.
Again, we start by considering $\supp(S)$, this time showing that $|\supp(S)|=4$.
Since $\so(G/H)= 4m-3$ and $|S|= (2n-2)m + 4m-4$. We know that there exist subsequences $S_1, \dots, S_{2n-2}$ such that $\prod_{i=1}^{n-1} S_i \mid S$ and each $\varphi(S_i)$ is a zero-sum sequence of length $m$. Let $R\in \fc(G)$ such that $S= R \prod_{i=1}^{2n-2} S_i$.
We note that $\s(S_1)\dots \s(S_{2n-2})$ has no zero-sum subsequence of length $n$. Thus, by Theorem \ref{thm_invcyc} we know that it is equal to $(e'(e'+e))^{n-1}$ where $\langle e \rangle = H $, say $\s(S_i)=(e'+e)$ for $i \in [1, n-1]$.

Moreover, $\varphi(R)$ has no zero-sum subsequence of length $m$. Since $|R|= 4m-4$ and $C_m^2$ has Property \textbf{D}, $\varphi(R)=T^{m-1}$ for some $T \in \fc(G/H)$.
Analogously to the proof of Theorem \ref{thm_main}.1, using that $C_m^2$  has Property \textbf{D}, it can be seen that if, for some $i \in [1, 2n-2]$, $\varphi(S_i)\notin \{f^m\colon f \in \supp(T)\}$, then $\varphi(RS_i)$ is divisible by the product of two zero-sum sequences of length $m$, yielding a contradiction. Thus, $\supp(\varphi(S))= \supp (\varphi (R))$ and  each $\varphi(S_i)$ is equal to $f^m$ for some $f \in \supp(T)$.

Now, we show that $|\varphi^{-1}(f)\cap \supp (S)|= 1$ for each $f \in \supp(\varphi(S))$.
Assume not. If $f \in \supp(\varphi(\prod_{i=1}^{2n-2}S_i))$, then this can be seen similarly to the proof of Theorem \ref{thm_main}.1.
Suppose $g, g' \in \supp(R)$ are distinct but $\varphi(g)= \varphi(g')=f$. This implies $m\ge 3$.
The sequence $g^{-1}R$ has a  subsequence $U$ of length $2m$ such that $\s(\varphi(U))=0$; this follows by Theorem \ref{thm_dir}, since in view of $\so_{m \N}(C_m^2)= 3m-2$, we get a sequence of length $m$ or $2m$, and by assumption it cannot have length $m$.
By Lemma \ref{lem_2m}, we may assume that it contains $g'$.
We show that $\s(U)= 2e'+e$. Assume not, say $\s(U)=2e'-ae$ with $a\in [0,n-2]$. We consider $US_1 \dots S_{a}S_{n} \dots S_{2n-a-3}$. The sum of this sequence is $(2e'-ae) + a(e'+e)+ (n-a-2)e'=0$ and its length is $2m + a m + (n-a-2)m=mn$, a contradiction.
Yet, by the same argument $\s(g'^{-1}gU)=2e'+e$, a contradiction.

Thus, we know $S= g^{s_0m-1}h_1^{s_1m -1}h_2^{s_2 m -1}h_3^{s_3m -1}$ with $s_i \in [1, n]$ and $s_0+s_1+s_2+s_3=2n+2$.

Without restriction we assume that $s_0$ is maximal. We have $s_0 m-1\ge m(n+1)/2 -1 \ge  \lfloor (mn-1)/2 \rfloor$.
Thus, by Lemma \ref{lem_general}, $(S-g)= 0^{s_0m-1} R T$ where $T$ is a sequence of length $mn +2m -3$ with no short zero-sum subsequence.
By Theorem \ref{thm_main}.1 we know all possible structures of $T$. It remains to determine $s_0$ and $R$.
If $T=g_1^{m-1} g_2^{mn -1} (-g_1+g_2)^{m-1}$ for some generating set $\{g_1,g _2\}$, we get, since $s_0$ is maximal,  that $s_0= n$ and
consequently $R=1$, implying  that $S$ is of the form given in (b).

Thus, it remains to consider   $T=e_1^{m-1} e_2^{km -1} (-xe_1+e_2)^{(n-k)m-1}$ for a basis $\{e_1, e_2\}$.
We note that if $\vo_{e_2}(RT)+\vo_{-xe_1+e_2}(RT)\ge mn + m -1$, then
$\lfloor \vo_{e_2}(RT)/m \rfloor + \lfloor\vo_{-xe_1+e_2}(RT)/m \rfloor \ge n$.
Yet, if this is the case, then $RT$ has a zero-sum subsequence of length $mn$, implying that $(g+ RT) \mid S$ has a zero-sum subsequences of length $mn$.
Consequently $\vo_{e_2}(RT)+\vo_{-xe_1+e_2}(RT)\le mn + m -2$ and thus $R= e_1^{(n-s_0)m}$, implying that $S$ is of the form given in (a). This completes the proof of part 2.

\subsection*{Funding}
This work is supported by the Austrian Science Fund FWF [grant numbers P18779-N13 and J2907-N18].


\end{document}